\documentclass[12pt, reqno]{amsart}

\usepackage{amssymb, amsmath, amsthm}
\usepackage[backref]{hyperref}
\usepackage[alphabetic,backrefs,lite]{amsrefs}
\usepackage{amscd}   
\usepackage{fullpage}
\usepackage[all]{xy} 
\usepackage{verbatim}

\DeclareFontEncoding{OT2}{}{} 

\usepackage{color}

\newtheorem{lemma}{Lemma}[section]
\newtheorem{theorem}[lemma]{Theorem}
\newtheorem*{theorem*}{Theorem}

\newtheorem{prop}[lemma]{Proposition}
\newtheorem{cor}[lemma]{Corollary}

\newtheorem{claim*}{Claim}
\newtheorem{thm}[lemma]{Theorem}

\newtheorem{notation}[lemma]{Notation}

\newtheorem{prob}{Problem}

\theoremstyle{definition}

\newtheorem{remark}[lemma]{Remark}

\newcommand{\Q}{{\mathbb Q}}

\newcommand{\N}{{\mathbb N}}
\newcommand{\Z}{{\mathbb Z}}

\newcommand{\calG}{{\mathcal G}}
\newcommand{\calH}{{\mathcal H}}

\newcommand{\calL}{{\mathcal L}}

\DeclareMathOperator{\Aut}{Aut}
\DeclareMathOperator{\Gal}{Gal}

\DeclareMathOperator{\id}{id}

\newcommand{\isom}{\cong}

\numberwithin{equation}{section}
\numberwithin{table}{section}

\title{First-order theory of a  field and its Inverse Galois Problem}
\author{Francesca Balestrieri}
\address{The American University of Paris \\ 5 Boulevard de La Tour-Maubourg\\
75007 Paris, France}
\email{fbalestrieri@aup.edu}
\author{Jennifer Park}
\address{The Ohio State University \\ Department of Mathematics \\ Columbus, OH 43210, USA}
\email{park.2720@osu.edu}
\urladdr{\url{https://u.osu.edu/park-2720/}}

\author{Alexandra Shlapentokh}
\address{East Carolina University\\ Department of Mathematics\\ Greenville, NC 27858, USA}
\email{shlapentokha@ecu.edu}
\urladdr{\url{http://myweb.ecu.edu/shlapentokha/}}
\date{\today}

\begin{document}
\maketitle
\begin{abstract}
Let $G$ be a finite group.  Then there exists a first-order statement $S(G)$ in the language of rings without parameters and depending only on $G$ such that, for any field $K$, we have that $K\models S(G)$ if and only if $K$ has a Galois extension with the Galois group isomorphic to $G$.  Further, there is an effective procedure which takes the table of multiplication of $G$ as its input and produces $S_G$. Therefore, given a field $K$, the Inverse Galois Problem for $K$, that is, the problem of deciding whether $K$ has a Galois extension with a particular Galois group  as input, is Turing reducible to the first-order theory of $K$. 
Similar results hold for the Finite Split Embedding Problem and the Inverse Automorphism Problem.
\end{abstract}


\section{Introduction}
This paper discusses the connection between the \emph{Inverse Galois Problem} (IGP) and the first-order theory of fields.  We state the IGP below.
\begin{prob}[Inverse Galois Problem over a field $K$ for a group $G$, IGP($K,G$)]
Let $K$ be a field, and let $G$ be a finite group. Does there exist a Galois extension $L$ of $K$ of degree $d$ with the Galois group of the extension isomorphic to $G$?
\end{prob}

This problem  has a long history dating back to the Kronecker-Weber theorem (which corresponds, in our notation, to IGP($\Q$, $\Z/n\Z$)) from the early XIX century. While many special cases of this problem and its variants are known \cite{Har87}, \cite{Koe04}, \cite{Pop96}, the general question remains open.

The first-order language of rings is the language $\calL_{R}=(0,1, +, \times)$ using universal and existential quantifiers.  The problem of deciding which statements of the language are true over a particular  ring $R$ can be reduced to the question whether the following type of statements is true:
\[
E_1 x_1\ldots E_k x_k P(x_1,\ldots,x_k)=0,
\]
where $E_i$ is either a universal or an existential quantifier ranging over $R$ and $P(x_1,\ldots,x_k)$ is a polynomial with coefficients in the ring.  

The fact that the Inverse Galois Problem for a field and a specific group can be coded into a first-order theory of a ring is apparently known (\cite{Koe04}), though, to the best of the authors' knowledge, an explicit specific statement about this does not exist in the literature, nor is it established how uniform such a statement would be. In this paper, we produce a statement $S(G)$, where $G$ is a finite group, in the first-order language of rings $\calL_R$ without any parameters such that for any field $K$ we have that $K\models S(G)$ if and only if the Inverse Galois Problem over K is solvable with Galois group $G$.  The sentence $S(G)$ depends only on the multiplication table of $G$ and there is an effective procedure to construct $S(G)$ from the multiplication table of the group.

We also construct first-order statements coding generalizations of the IGP, namely the \emph{Finite Split Embedding Problem} (FSEP) and what we call the \emph{Inverse Automorphism Problem} (IAP).

\subsection{Turing degrees and reduction}
We first define Turing degrees and give a precise formulation of the problem we want to address. (For a general reference on Turing degrees see, for example, \cite{Soare87}).

\subsubsection{Computable and listable sets and Turing degrees}
A subset $S \subset \Z$ is \textbf{computable} if there exists an algorithm (or a computer program terminating on every input) that determines membership in the set.
A subset $S \subset \Z$ is \textbf{listable} if there exists an algorithm (or a computer program) that lists the set.  Given two subsets $A, B \subset \Z$, we say that $A \leq_T B$ ($A$ is Turing reducible to $B$) if there exists an algorithm taking the characteristic function of $B$ as its input and generating the characteristic function of $A$. If $A\leq_T B$ and $B \leq_T A$, then we say that $A \equiv_T B$ ($A$ is Turing equivalent to $B$).  The relation $\equiv_T$ is an equivalence relation and the corresponding equivalence classes are called \emph{Turing degrees}.


\subsubsection{Encoding the collection of the isomorphism classes of finite groups and the theory of a field} \label{Hld}
We identify the set of isomorphism classes of finite groups with $\N$ via a map $\sigma$ defined in the following way:
\begin{itemize}
    \item Let $\calG$ be a collection of representatives of isomorphism classes of finite groups (note that $\calG$ is countable, so there is an effective bijection $\sigma: \calG \to \N$ with the image being exactly $\N$.) In other words, $\sigma $ attaches a natural number as a label to each (isomorphism class of a) finite group.
    \item We note that, if $G$ is some finite group, then  $\sigma(G)$ can be effectively determined once we are given a multiplication table of a group $G$ and, vice versa, we can effectively recover a multiplication table for $G$ from $\sigma(G)$. It is clear that this only depends on the isomorphism class of $G$ in $\calG$.
\end{itemize}
Let $K$ be any fixed field. For each $d \in \N$, let 
\[
\calH_{K,d} = \{n \in \N: \# \sigma^{-1}(n) =d  \textup{ and $K$ has a Galois extension with Galois group $\sigma^{-1}(n)$}\}.
\]
In other words, $\calH_{K,d}$ is the set of all labels corresponding under $\sigma$ to (isomorphism classes of) groups $G$ of size $d$ such that there is a Galois extension of $K$ with Galois group isomorphic to $G$.

\subsubsection{Main theorem} In this paper we investigate the relation between the first-order theory of fields and the Inverse Galois Problem. As a consequence, we are also able to study the Turing degree of $\calH_{K,d}$, by comparing the Turing degree of $\calH_{K,d}$ to the Turing degree of the first-order theory of $K$ in the language of rings.  Recall that the \emph{first-order theory  of a field $K$ in the language of rings}, denoted $\rm{Th}(K)$,  is the collection of all sentences in the language of rings (without parameters or additional constants)  that are true over $K$.   

Since we are using a countable language and since a sentence is a finite string of symbols of the language, the collection of all sentences in the language is countable and can be put into an effective bijection with $\N$.  Thus we can identify $\rm{Th}(K)$ with a subset of $\N$ and define the Turing degree of the theory to be the Turing degree of the corresponding subset of $\N$.

Using definitions above one of the main theorems of the paper can be stated as follows.

\begin{theorem*}(Corollary \ref{maincor})
\begin{enumerate}
\item For every finite group $G$ there exists a statement $S(G)$ in the first-order language of rings without parameters or additional constants such that for any field $K$ we have that $K \models S(G)$ if and only if $IGP(K,G)$ has a solution.  In other words $S(G)$ is an axiomatization of fields $K$ with solvable $IGP(K,G)$.
\item The statement $S(G)$ depends only on the multiplication table of $G$ and there exists an effective procedure taking the table of multiplication as its input and producing $S(G)$.
    \item For any field $K$ we have that $\calH_{K,d} \leq_T \rm{Th}(K)$ and the Turing reduction (i.e. the construction of the sentence $S(G)$) is independent of $K$.
\end{enumerate}
\end{theorem*}

\subsubsection{Variations on the main theorem}
By similar methods, we also prove several variations of IGP, an example of which is stated below.

\begin{theorem*}(Theorem \ref{thm:conj})
Let $H$ be a finite group of size $m$.  Let $n$ be a multiple of $m$.  Then there exists a first-order statement $\varepsilon_{H,n}(x_1,\ldots,x_n)$ in the language of rings without parameters depending on $H$ and $n$ only such that $F\models \exists \bar x \in F^n: \varepsilon_{H,n}(x_1,\ldots,x_n)$ if and only if there exists an extension $L/F$ of degree $n$ such $\Aut(L/F) \cong H$.
\end{theorem*}


\section*{Acknowledgments}
F.B. and J.P. thank the Institut Henri Poincar\'e and the organizers of the trimester \textit{\`A la r\'edecouverte des points rationnels} for allowing this project to begin. J.P. and A.S. thank American Institute of Mathematics for creating an environment where their collaboration could start.  The authors thank Arno Fehm and Laurent Moret-Bailly for helpful comments.  F.B. was partially supported by the European Union's Horizon 2020
research and innovation programme under the Marie Sk\l{}odowska-Curie grant 840684. J.P. was partially supported by NSF DMS-1902199 and DMS-2152182, A.S. was partially supported by NSF DMS-2152098. 




\section{The first-order theory of a field and its Inverse Galois Problem} \label{S:IGP}

Our goal is to prove the following theorem.

\begin{theorem}
\label{thm:irr}
 Let $G$ be a finite group of size $d$. 
Then there exists an effective procedure taking as the input the multiplication table of the group and  constructing a first-order statement $S(G)$ in the language of rings such that, for any field $K$, we have that $K \models S(G)$ if and only if $K$ has a Galois extension with the Galois group isomorphic to $G$.
\end{theorem}

\begin{remark}
For any base field under consideration, we fix once and for all an algebraic closure of that field, and we consider any extension that we construct to be inside this algebraic closure.
\end{remark}

\begin{remark}
For a given field $K$, there always exists a first-order sentence  such that it is true in $K$ if and only if $K$ has a Galois extension with Galois group isomorphic to $G$.  One could let the statement be a tautology if the extension exists, and the negation of the tautology otherwise.  However, this statement clearly depends on $K$.  The point of the theorem above is that there is an algorithm to construct a first-order sentence so that for {\it any field} $K$ we have that the statement is true over $K$ if and only if $K$ has a requisite Galois extension.
\end{remark}
The following corollary is an obvious consequence of Theorem \ref{thm:irr}. 
\begin{cor}
The Turing degree of IGP over any field  is less than or equal to the degree of the first-order theory of the field.
\end{cor}

For the remainder of the paper we use the following notation.
\begin{notation} 
\begin{itemize}
    \item $K$ is a field.
     \item For a $n$-tuple $\bar a := (a_0, \ldots, a_{n-1})$, let 
$f_{\bar a}(T) := a_0 +a_1 T+\ldots +a_{n-1}T^{n-1}+T^n$.
   \item $I_{K,n} \subset K^n$ is such that $\bar a := (a_0,\ldots, a_{n-1}) \in I_{K,n}$ if and only if the polynomial $f_{\bar a}(T)$ is  irreducible over $K$.
  
\end{itemize}
\end{notation}
In order to talk about extensions of a field we need a way to describe the irreducible (over the field) polynomials. This is not a hard task if we are using the full first-order (as opposed to existential) language.
\begin{prop}
\label{prop:definable}
For each $n \in \N$ there exists a first-order formula  $\phi_n(x_0,\ldots, x_{n - 1})$ in the language of rings without parameters where $x_0,\ldots,x_{n-1}$ are the only free variables and such that, for any field $K$ and any $\bar a :=(a_0,\ldots,a_{n-1}) \in K^n$, we have that $K \models \phi_n(a_0,\ldots,a_{n-1})$ if and only if the polynomial $f_{\bar a}(T)$ is irreducible over $K$.
\end{prop}

\begin{proof} The case for $n=1$ is clear, so we may assume that $n > 1$.
The following existential formula $F_{d,n}(x_0,\ldots,x_{n-1})$, for $d = 1, ..., n-1$, is such that  for any field $K$ and any  $(a_0,\ldots,a_{n-1}) \in K^n$ we have that $K \models F_{d,n}(a_0,\ldots,a_{n-1})$ if and only if the polynomial $f_{\bar a}(T)$ can be factored into a degree $d$ factor and a degree $n-d$ factor:
\begin{equation}
\label{eq:factor}
\begin{array}{l l}
  (F_{d,n}(a_0,\ldots,a_{n-1})) &\exists x_{0,d}, \ldots, x_{d,d}, y_{0,d}, \ldots, y_{n-d,d} \in K: \\ 
   & f(T) = (x_{0,d} + x_{1,d}T + \cdots + x_{d,d}T^d)(y_{0,d} + y_{1,d}T+ \cdots + y_{n-d,d}T^{n-d}).
\end{array}
\end{equation}

Therefore, for any $K$ and any $(a_0,\ldots,a_{n-1}) \in K^n$ we have that $K \models \bigvee_{i=1}^{n-1}F_{d,n}(a_0,\ldots,a_{n-1})$ if and only if $f_{\bar a}(T)$ is reducible over $K$.
Hence, for any field $K$ and any $(a_0,\ldots,a_{n-1}) \in K^n$ we have that $ K\models \lnot (\bigvee_{i=1}^{n-1}F_{d,n}(a_0,\ldots,a_{n-1}))$ if and only if $f_{\bar a}(T)$ is irreducible.
\end{proof}

We now consider formulas which will guarantee existence of a Galois extension associated to a given polynomial.

\begin{thm}
\label{thm:galois1} For any $n \in \N$, there exists a  formula \[\psi_n(x_0,\ldots,x_{n-1})\] in the language of rings without parameters such that, for any field $K$ and any $a_0,\ldots,a_{n-1} \in K$, we have that 
\[
K\models\psi_n(a_0,\ldots,a_{n-1})
\]
if and only if $f_{\bar a}(T)$ is irreducible over $K$ and any of its roots generates a Galois extension of $K$.
\end{thm}
\begin{proof} 
Consider the following formula $\psi_n(x_0,\ldots,x_{n-1})$, where $T$ is transcendental over $K$ and is used for convenience:
\[ 
 \exists y_{1,0}, \ldots, y_{1,n-1}, \ldots,  y_{n,0}, \ldots, y_{n,n-1}  : 
\]
\begin{equation}
\label{sys:1}
\left \{
\begin{array}{l}
\phi_n(x_0,\ldots,x_{n-1}) \\
f_{\bar x}(\sum_{j=0}^{n-1}y_{i,j}T^{j}) \equiv 0 \bmod f_{\bar x}(T), \textrm{ for all }  i\in \{1,\ldots, n\} \\
\sum_{j=0}^{n-1} y_{i,j} T^j \not \equiv \sum_{j=0}^{n-1} y_{r,j} T^j \bmod f_{\bar x}(T) \text{ for all $i,r \in \{1, \ldots, n\}$} \text{ with } i \ne r,
\end{array}
\right .
\end{equation}
where $\phi_n(x_0, ..., x_{n-1})$ is the formula from Proposition \ref{prop:definable}.

We now show that $\psi_n(x_0,\ldots,x_{n-1})$ satisfies the conditions described in the proposition.  First assume that for some $a_0,\ldots,a_{n-1} \in K$ we have that $K\models \psi_n(a_0,\ldots,a_{n-1})$. Then $\phi_n(a_0,\ldots,a_{n-1})$ is true and the polynomial $f_{\bar a}(T)$ is irreducible.  Consider an isomorphism $K[T]/(f_{\bar a}(T)) \rightarrow K(\alpha)$, sending $T$ to $\alpha$, where $\alpha$ is a root of $f_{\bar a}(T)$ in $\bar{K}$, the fixed algebraic closure of $K$.  Let $b_{1,0}, \ldots, b_{n,n-1} \in K$ be the elements satisfying the last two equations of \eqref{sys:1} in place of the $y_{i,j}$.  The second equation of \eqref{sys:1} tells us that $\sum_{j=0}^{n-1}b_{i,j}\alpha^j$ is a root of $f_{\bar a}(T)$ for each $i\in \{1,\ldots,n\}$.  Hence, for each $i \in \{1,   \ldots, n\}$ we have that $\sum_{j=0}^{n-1}b_{i,j}\alpha^j$ is a root of $f_{\bar a}(T)$ in $K(\alpha)$.  Finally, the third equation of \eqref{sys:1} implies that all the listed roots are distinct.  Hence, in $K(\alpha)$ the polynomial $f_{\bar a}(T)$ splits and therefore $K(\alpha)/K$ is a Galois extension.

We assume now that $f_{\bar a}(T)$ is irreducible, any root of $f_{\bar a}(T)$ in $\bar K$ generates a Galois extension of $K$,  and show that $K\models \psi_n(a_0,\ldots,a_{n-1})$.   

Since $f_{\bar a}(T)$ is irreducible we have that $K \models\phi_n(a_0,\ldots,a_{n-1})$.  Let $\alpha \in \bar K$ be any root of $f_{\bar a}(T)$ and once again consider an isomorphism $K[T]/(f_{\bar a}(T)) \xrightarrow{\sim} K(\alpha)$, sending $T$ to $\alpha$.  Let $\alpha_1=\alpha, \ldots, \alpha_n$ be all the conjugates of $\alpha$ over $K$ or, in other words, all the roots of $f_{\bar a}(T)$ in $\bar K$.  Since for any $i=1,\ldots, n$ the extension $K(\alpha_i)/K$ is Galois, for every $i=1,\ldots, n$ there exist $s_{i,j}$ with $ j=0,\ldots, n-1$ such that $\alpha_i=\sum_{j=0}^{n-1}s_{i,j}\alpha^j$.  Therefore if we set $y_{i,j}:=s_{i,j}$  we will have that 
\[
f_{\bar a}\left(\sum_{j=0}^{n-1}y_{i,j}T^{j}\right) \equiv 0 \bmod f_{\bar a}(T), \textrm{ for all }  i\in \{1,\ldots, n\}.
\]
Further, since the extension $K(\alpha)/K$ is separable, we have that $\alpha_i \ne \alpha_r$ and with $y_{i,j}:=s_{i,j}$ we will have that 
\[
\sum_{r=1}^{n-1} y_{i,j} T^j \not \equiv \sum_{r=1}^{n-1} y_{r,j} T^j \bmod f_{\bar a}(T) \text{ for all $i,r \in \{1, \ldots, n\}$} \text{ with } i \ne r.
\]

Thus, we conclude that there exist $y_{1,0}, \ldots, y_{1,n-1}, \ldots,  y_{n,0}, \ldots, y_{n,n-1}$ satisfying the formulas in \eqref{sys:1}.
\end{proof}

\begin{thm} 
\label{thm:Galois}
Let $n \in \mathbb{Z}_{\geq 0}$, and let $G$ be a finite group of order $n$.  There  exists a first-order formula $\theta_{G}(x_0,\ldots,x_{n-1})$ in the language of rings without parameters such that, for any field $K$ and any $a_0,\ldots, a_{n-1} \in K$, we have that $K\models \theta_{G}(a_0,\ldots,a_{n-1})$ if and only if the polynomial $f_{\bar a}(T)$ is irreducible and every root of $f_{\bar a}(T)$ in $\bar K$ generates a Galois extension of $K$ with Galois group isomorphic to $G$. Further, $\theta_G$ is effectively constructible from the table of multiplication of $G$.    
\end{thm}

\begin{proof}
Let $G$ be a group of order $n$. Let $\sigma_i$ for $i=1,\ldots, n$ be the elements of the group and let 
\begin{equation}
\label{eq:table}
\sigma_{i}\sigma_r = \sigma_{\rho_{i,r}}, \quad i,r \in \{1, \ldots, n\},
\end{equation}
where $\rho_{i,r} \in \{1, \ldots, n\}$ is uniquely determined by $i$ and $r$.  In other words, we are given a map $\rho: \{1,\ldots,n\}\times \{1,\ldots,n\} \longrightarrow \{1,\ldots,n\}$ describing the binary operation of $G$ and we denote $\rho(i,r)$ by $\rho_{i,r}$.

  Next consider the following conjunction $\theta_{G}(x_0,\ldots,x_{n-1})$:

\[
 \exists y_{1,0}, y_{1,1}, \ldots, y_{1,n-1},y_{2,0,} \ldots,  y_{n,0}, \ldots, y_{n,n-1}  : 
\]
\begin{equation}
\label{eq:group}
\left \{
\begin{array}{l}
\phi_n(x_0,\ldots,x_{n-1}) \\
f_{\bar x}(\sum_{j=0}^{n-1}y_{i,j}T^{j}) \equiv 0 \bmod f_{\bar x}(T), \textrm{ for all }  i\in \{1,\ldots, n\} \\
\sum_{j=0}^{n-1} y_{i,j} T^j \not \equiv \sum_{j=0}^{n-1} y_{r,j} T^j \bmod f_{\bar x}(T) \text{ for all $i,r \in \{1, \ldots, n\}$} \text{ with } i \ne r \\
\sum_{j=0}^{n-1}y_{i,j}\left(\sum_{m=0}^{n-1}y_{r,m}T^m\right)^j\equiv \sum_{\ell=0}^{n-1}y_{\rho_{i,r},\ell}T^{\ell} \bmod f_{\bar x}(T), \text{ for all } i,r \in \{1,\ldots,n\},
\end{array}
\right .
\end{equation}
where $f_{\bar x}(T):= T^n + x_{n-1}T^{n-1} + ... + x_1 T + x_0$ and where  $\phi_n(x_0, ..., x_{n-1})$ is the formula from Proposition \ref{prop:definable}. We observe that the first three equations are those defining $\psi_n(x_0, ..., x_{n-1})$ from Theorem \ref{thm:galois1}.\\

We claim that for any field $K$, any finite group $G$, and $n$-tuple $(a_0,\ldots,a_{n-1}) \in K^n$, we have that $K\models \theta_G(a_0,\ldots,a_{n-1})$ if and only if the polynomial $f_{\bar a}(T) $ is irreducible and any of its roots generates a Galois extension of $K$ with  Galois group isomorphic to $G$.  

First suppose that $K\models \theta_G(a_0,\ldots,a_{n-1})$.  Then $K \models  \psi_n(a_0,\ldots,a_{n-1})$, implying in particular that $f_{\bar a}(T)$ generates a Galois extension of degree $n$, where $n$ is the order of $G$.  Let $\alpha$ be a root of $f_{\bar a}(T)$.  Then $K(\alpha)/K$ is a Galois extension of $K$ of degree $n$.  Let $\alpha_i=\sum_{i,j}b_{i,j}\alpha^j$, where $y_{1,0}:=b_{1,0},\ldots, y_{n,n-1}:=b_{n,n-1} \in K$ and $b_{1,0},\ldots, b_{n,n-1}$ are the elements of $K$ whose existence is guaranteed by \eqref{eq:group}.  As in the proof of Theorem \ref{thm:galois1}, we have that $f(\alpha_i)=0$ for $i \in \{1,\ldots, n\}$ and that $\alpha_i \ne \alpha_j$ for $i\ne j$.

Given some enumeration of elements of $G=\{\sigma_1,\ldots,\sigma_n\}$ we can let $\sigma_i(\alpha)=\alpha_i$.  Then the table of multiplication of $G$ corresponds to the following equations
\begin{equation}
\label{sys:4}
\begin{array}{ll}

\sigma_r(\sigma_i(\alpha)) &= \sigma_r(\alpha_i)\\
\vspace{0.2 cm}
& \displaystyle= \sigma_r(y_{i,0} + y_{i,1}\alpha + \cdots + y_{i,{n-1}}\alpha^{n-1}) \\
&\displaystyle = (y_{i,0} + y_{i,1}\alpha_r + \cdots + y_{i,{n-1}}\alpha_r^{n-1})\\
&\displaystyle = \left( y_{i,0} + y_{i,1}\left(\sum_{m=0}^{n-1} y_{r,m}\alpha^m \right) + \cdots + y_{i,{n-1}}\left(\sum_{m=0}^{n-1} y_{r,m}\alpha^m\right)^{n-1}\right)\\
 = \sigma_{\rho_{i,r}}(\alpha)
&\displaystyle =(y_{\rho_{i,r},0} + y_{\rho_{i,r},1}\alpha + \cdots + y_{\rho_{i,r}, n-1}\alpha^{n-1}).
\end{array}
\end{equation}
Observe that, under the isomorphism $K[T]/(f_{\bar a}(T)) \longrightarrow K(\alpha)$ sending $T$ to $\alpha$, the equations from \eqref{sys:1} are equivalent to equations \eqref{eq:group} with $y_{i,j}:=b_{i,j}$.  Thus, $K \models \theta_G(a_0,\ldots,a_{n-1})$ implies that any root of $f_{\bar a}(T)$ generates a Galois extension of $K$ with the Galois group isomorphic to $G$.

Suppose now that, for a polynomial $f_{\bar a}(T):= T^n + a_{n-1}T^{n-1} + ...+ a_0 \in K[T]$, we have that any root of $f_{\bar a}(T)$ generates a Galois extension with the Galois group isomorphic to $G$.  We show that $K \models \theta_G(a_0,\ldots,a_{n-1})$.  Since $f_{\bar a}(T)$ generates a Galois extension, we have that $K\models \psi_n(a_0,\ldots,a_{n-1})$  by Theorem \ref{thm:galois1}.  Therefore  there exist $b_{1,0}, \ldots, b_{n,n-1} \in K$ such that, under the assignment $y_{1,0}:=b_{1,0}, ..., y_{n, n-1}:=b_{n, n-1}$, the first three equations of \eqref{eq:group} are satisfied.  It remains to be shown that the fourth equation is satisfied with the assignment of values from $K$.

Let $\alpha$ be  the image of $T$ under the isomorphism of $K[T]/f_{\bar a}(T)$ and $K(\alpha)$.  Let $\sigma_i(\alpha):=\alpha_i:=\sum_{i,j}b_{i,j}\alpha^j,$ for $ i\in \{1,\ldots, n\}$ and $ j\in \{0, \ldots, n-1\}$. Then the equation $\sigma_r\circ \sigma_i=\sigma_{\rho_{i,j}}$ implies that \eqref{sys:4} holds and, by the equivalence described above, \eqref{eq:group}  holds.
\end{proof}

The theorem above provides an explicit link between IGP($K,G$) and Th($K$). Recall from Section \ref{Hld} that $\calH_{K,n}$ is in bijection with the collection of all groups $G$ of size $n$ for which a Galois extension of $K$ with Galois group $G$ exists. Thus we have the following corollary.

\begin{cor}\label{maincor}
Let $K$ be a countable field. Then the following are true.
\begin{enumerate}
\item  Let $G$ be a finite group of size $n$. Then there exists a first-order sentence $\lambda_G$, dependent on $G$ only, such that $K \models \lambda_G$ if and only if $\textrm{\normalfont IGP}(K,G)$ has a solution.
\item  $\calH_{K,n} \leq_T \textrm{\normalfont Th}(K)$ and the reduction is uniform in $K$. 
\item If $\textrm{\normalfont Th}(K)$ is decidable, then there is an algorithm (possibly dependent on $K$) to decide whether $\textrm{\normalfont IGP}(K,G)$ is true, for any finite group $G$.
\end{enumerate}
\end{cor}
\begin{proof}
\begin{enumerate}
   \item Let $\lambda_G$ be the first-order sentence $\exists x_0, \ldots, x_{n-1}: \theta_G(x_0,\ldots,x_{n-1})$.  Suppose $\text{IGP}(K,G)$ has a solution.  Then there exists $f_{\bar a}(T) \in K[T]$,  a monic irreducible polynomial such that its root generates a Galois extension $L/K$ with $\Gal(L/K) \isom G$.  In this case, by Theorem \ref{thm:Galois}, we have that $K\models \theta_G(a_0,\ldots,a_{n-1})$ and consequently $K \models \lambda_G$.

    Conversely, suppose $K \models \lambda_G$.  Then by Theorem \ref{thm:Galois} there exist $a_0,\ldots,a_{n-1} \in K$ such that $K \models \theta_G(a_0,\ldots,a_{n-1})$ implying that $f_{\bar a}(T)$ is irreducible over $K$ and its root generates a Galois extension of $K$ with the Galois group isomorphic to $G$.
 \item From Part (1) of the Corollary, it follows that to determine whether the label for the finite group $G$ of size $n$ given by $\sigma$ (see Section \ref{Hld})  is in $\calH_{K,n}$, it is enough to determine whether $\lambda_G$ is true over $K$.  Since the procedure for constructing $\lambda_G$ is independent of $K$, the reduction $\calH_{K,n} \leq_T \textrm{Th}(K)$ is independent of $K$.
 \item The last part follows from Part (2) of the Corollary. \qedhere
 \end{enumerate}
\end{proof}

\begin{remark}
Note that if $\textrm{\normalfont Th}(K)$ is decidable and $\textrm{\normalfont IGP}(K,G)$ is true, then there is an algorithm to explicitly construct a polynomial with coefficients in $K$ producing the required extension. We need only to systematically check all $n$-tuples of elements of $K$ to find one that corresponds to an irreducible polynomial and such that the resulting system of equations has solutions in other variables in $K$.  
\end{remark}
We now prove a slightly different version of Theorem \ref{thm:Galois} that we will need in the future.

\begin{prop}
\label{prop:tower}
    Let $n, m \in \mathbb{Z}_{\geq 0}$, and let $H, G$ be finite groups of orders $n$ and $m$ respectively.  There  exists a first-order formula $\theta_{H,G}(x_0,\ldots,x_{n-1},y_{0,0}, \ldots, y_{m-1,n-1})$ in the language of rings without parameters such that, for any field $F$ (with algebraic closure $\overline{F}$) and any $a_0,\ldots, a_{n-1}, c_{0,0}, \ldots, c_{n-1,m-1} \in F$, we have that $F\models \theta_{H,G}(a_0,\ldots,a_{n-1},c_{0,0}, \ldots, c_{n-1,m-1})$ if and only if
    \begin{enumerate}
        \item  the polynomial $f_{\bar a}(T)$ is irreducible over $F$ and every root of $f_{\bar a}(T)$ in $\bar F$ generates a Galois extension of $F$ with Galois group isomorphic to $H$; and
        \item if $\beta$ is any root of $f_{\bar a}$ in $\bar F$ and $\bar b=(b_0,\ldots,b_{m-1})$ is such that $b_i=\sum_{j=0}^{n-1}c_{i,j}\beta^j$ for $ i=0,\ldots, m-1$,  then $f_{\bar b}$ is irreducible over $F(\beta)$ and any root of $f_{\bar b}$ generates a Galois extension $F(\gamma, \beta)$ of $F(\beta)$ with the Galois group isomorphic to $G$.
    \end{enumerate}
    
    Further, $\theta_{H,G}$ is effectively constructable from the tables of multiplication of $G$ and $H$.    
\end{prop}
\begin{proof}
    To facilitate the discussion below we let $b_m=1$, $c_{m, j} = 0$ for $j \neq 0$ and $c_{m, 0}=1$. We start with modifying Proposition \ref{prop:definable} for the field $F(\beta)$, where  $\beta$ is any root of $f_{\bar a}$. 
    For any $d = 1, ..., m-1$, we let $R_{d,n, m}(\bar a, \bar b,x_{0,0,1},\ldots,x_{d,n-1,d})$ denote the following formula:
    \begin{equation}
\label{eq:factor1}
\begin{array}{c}
    \exists  y_{0,0,m-d}, \ldots, y_{m-d,n-1,m-d}: \\ 
   \displaystyle f_{\bar b}(U) = \sum_{i=0}^{m}b_iU^i=\sum_{i=0}^{m}\sum_{j=0}^{n-1}c_{i,j}\beta^jU^i=
   \left (\sum_{i=0}^d\sum_{j=0}^{n-1}x_{i,j,d}\beta^jU^i\right)\left (\sum_{i=0}^{m-d}\sum_{j=0}^{n-1}y_{i,j,m-d}\beta^jU^i \right)
\end{array}
\end{equation}
encoding a system of equations over $F(\beta)$ we can construct by treating $U$ as transcendental over $F(\beta)$.  To get rid of $\beta$, as above we use a variable $T$, where $T$ and $U$ are algebraically independent over $F$, and replace equations by congruences modulo $f_{\bar a}(T)$:
\begin{equation}
\label{eq:factor2}
\begin{array}{c}
    \exists  y_{0,0,m-d}, \ldots, y_{m-d,n-1,m-d}: \\ 
   f_{\bar b}(U,T) = \sum_{i=0}^{m}b_iU^i=\sum_{i=0}^{m}\sum_{j=0}^{n-1}c_{i,j}T^jU^i \equiv\\
   \left (\sum_{i=0}^d\sum_{j=0}^{n-1}x_{i,j,d}T^jU^i\right)\left (\sum_{i=0}^{m-d}\sum_{j=0}^{n-1}y_{i,j,m-d}T^jU^i \right) \bmod f_{\bar a}(T).
\end{array}
\end{equation}
Similarly to what we had for $f_{\bar a}(T)$, the irreducibility of $f_{\bar b}(U)$ corresponds to the sentence $\pi_{n,m}(\bar a,c_{0,0}, \ldots, c_{m,n-1})$ given by
\begin{equation}
    \label{le:irreducible}
\forall x_{0,0,1},\ldots,x_{m-1,n-1,m-1} \bigwedge_{d=1}^{m-1}\lnot R_{d,n,m}(\bar a, c_{0,0}, \ldots, c_{m, n-1}, x_{0,0,d},\ldots,x_{d,n-1,d}).
\end{equation}
Let $G=\{\tau_1=\id,\ldots,\tau_m\}$ and let $\mu: \{1,\ldots, m\}\times \{1,\ldots, m\} \longrightarrow \{1,\ldots,m\}$ be the map representing multiplication table of $H$ meaning $\tau_i\circ \tau_j=\tau_{\mu_{i,j}}$.
Then 
\[
F \models \theta_{H, G}(a_0,\ldots, a_{n-1}, c_{0,0}, \ldots, c_{n-1,m-1})
\]
is equivalent to all the following conditions holding:

\begin{itemize}
    \item $F\models \theta_H(a_0,\ldots,a_{n-1})$, where a root of the irreducible polynomial $f_{\bar a}(T)$ generates a Galois extension over $F$ with Galois group $H$;
    \item $F \models \pi_{m,n}(\bar a,\bar c)$, meaning that $f_{\bar b}(U) \in F(T)[U]$ is irreducible with $b_i\equiv \sum_{j=0}^{n-1}c_{i,j}T^j \bmod f_{\bar a}(T)$ for all $i \in \{0, \ldots , m-1\}$;
    \item $F\models \exists v_{0,0,0}, \ldots, v_{m-1,m-1,n-1} \in F: $
    \begin{itemize} 
    \item $f_{\bar b}(\sum_{j=0}^{m-1}u_{i,j}U^{j}) \equiv 0 \bmod f_{\bar b}(U)$, for all $i\in \{1,\ldots, m\}$, where 
    \[
    u_{i,j}\equiv\sum_{r=0}^{n-1} v_{i,j,r}T^r \bmod f_{\bar a}(T);\]
    \item $\sum_{j=0}^{m-1} u_{i,j} U^j \not \equiv \sum_{j=0}^{m-1} u_{r,j} U^j \bmod f_{\bar b}(U)$ for all $i,r \in \{1, \ldots, m\}$ with  $i \ne r$;
    \end{itemize}
    \item $F\models \exists w_{0,0,0}, \ldots, w_{m-1,m-1,n-1} \in F$:
    \begin{itemize}
       \item $\tau_i(U)=:U_i\equiv\sum_{r=0}^{m-1}z_{i,r}U^r \bmod f_{\bar b}(U), $ 
       where      
\[
z_{i,r}\equiv \sum_{s=0}^{n-1}w_{i,r,s}T^s \bmod f_{\bar a}(T), 
\]
\item 

\[
\begin{array}{lll}
\tau_i\tau_j(U) & =\tau_i(U_j)\\
&=\tau_i(\sum_{r=0}^{m-1}z_{j,r}U^r)\\
&=\sum_{r=0}^{m-1}z_{j,r}U_i^r\\
&= \sum_{r=0}^{m-1}z_{j,r}(\sum_{s=0}^{m-1}z_{i,s}U^s)^r\\
&=\tau_{\mu_{i,j}}(U)\\
&\equiv U_{\mu_{i,j}} & \bmod f_{\bar b}(U)\\
&\equiv \sum_{s=0}^{m-1}z_{\mu_{i,j},s}U^s &\bmod f_{\bar b}(U). 
\end{array} 
\]
\qedhere
\end{itemize}
\end{itemize}
\end{proof}

\begin{prop}
\label{prop:biggroup}
    Let $H$ be a subgroup of the finite group $S$. Let $\pi \colon S\longrightarrow H$ be an epimorphism. Let $F$ be a field.    Then there exists a first-order statement $\nu_{F, H, S, \pi}$ in the language of rings such that $F \models \nu_{F, H, S, \pi}$ if and only if there exist Galois extensions $E \subset L$ of $F$ such that $\Gal(L/F) \cong S, \Gal(E/F)\cong H$ and $\pi$ is the natural restriction map $\Gal(L/F) \to \Gal(E/F)$.
\end{prop}
\begin{proof}
   By Proposition  \ref{thm:Galois} there exist first-order statements $\theta_H(\bar x)$ and $\theta_S(\bar y)$ such that 
   \[F \models \theta_H(\bar x) \land \theta_S(\bar y)\] 
   if and only if there exist irreducible polynomials $f_{\bar x}$ and $f_{\bar y}$ over $F$ such that the root of $f_{\bar x}$ generates a Galois extension $E/F$ with $\Gal(E/F) \cong H$, a root of $f_{\bar y}$ generates a Galois extension $L/F$ with $\Gal(L/F) \cong S$. Let $S=\{\sigma_1=\id,\ldots,\sigma_n\}$ and let $H=\{\mu_1=\id,\mu_2,\ldots, \mu_m\}$.  Then $\pi: \{1,\ldots, n\} \longrightarrow \{1,2,\ldots,m\}$.  Now consider the following conjunction of first-order statements:
   \[\begin{cases}
   \begin{array}{l}
   \theta_H(\bar x),\\
   \theta_S(\bar y),\\
   f_{\bar x}(\sum_{i=0}^nu_iT^i)\equiv 0 \bmod f_{\bar y}(T),\\
    U=\sum_{i=0}^nu_iT^i, \\
    \sigma_i(U)\equiv \mu_{\pi(i)}(U) \bmod f_{\bar y}(T), i=1,\ldots, n,
    \end{array}
    \end{cases}\]
   where equations specifying the action of $\sigma_i$ and $\mu_{\pi(i)}$ are part of $\theta_S$ and $\theta_H$ respectively, e.g. these are the equations \eqref{eq:group} for $S$ and $H$.
\end{proof}

We can also combine Propositions \ref{prop:tower} and \ref{prop:biggroup} to get the following proposition.
\begin{prop}
\label{prop:together}
    Let $H, G, S$ be finite groups such that $G \unlhd S$ and $H \cong S/G$.  Let $|S|=:s, |G|=:m, |H|=:r=s/m$. Let $F$ be any field.  Let $\bar a \in F^r,\bar b_i \in F^{r}, i=0,\ldots, m-1, \bar c \in F^s$.  Let $\pi: S \longrightarrow H$.  Then there exists a first-order statement  $F \models \mu_{H,G,S,\pi}(\bar a, \bar b_0, \ldots,\bar b_{m-1}, \bar c)$ if and only if the following statements are true.
    \begin{enumerate}
        \item A root $\alpha$ of the irreducible polynomial $f_{\bar a}(T)$generates a Galois extension $E$ of $F$ with $\Gal(E/F) \cong H$.
        \item If $d_i=\sum_{j=0}^{r-1}b_{i,j}\alpha^j, i=0,\ldots,m-1$, then $f_{\bar d}$ is irreducible over $E=F(\alpha)$ and any root $\beta$ of $f_{\bar d}$ generates an extension $L=E(\beta)$ such that $\Gal(L/E) \cong G, \Gal(L/F)\cong S$.
        \item Any root $\gamma$ of the irreducible polynomial  $f_{\bar c}(T)$ generates $L=F(\gamma)$.
        \item $\pi$ is the natural restriction map $\Gal(L/F) \to \Gal(E/F)$.
    \end{enumerate}
\end{prop}
\begin{remark}
    Equation \eqref{sys:4} provides a way of converting equations involving group elements into equations over a given field.  In other words, if $\sigma_i, \sigma_j, \sigma_r$ are elements of some group $G$, then the equation $\sigma_i \circ \sigma_j=\sigma_r$ holds if and only if the given field has a Galois extension $F$ generated by the roots $\alpha=\alpha_1,\ldots, \alpha_n$ of a polynomial $f$ over $F$ with $n$ equal to the size of the group, and under some ordering of the roots, there exists automorphisms of the extension sending $\alpha$ to $\alpha_i$ and sending $\alpha$ to $\alpha_j$ that if composed send $\alpha$ to $\alpha_r$.  In fact, given an $n\times n$ table which might or might not correspond to a table of multiplication of a group, one can establish whether the table corresponds to a group $G$ such that $F$ has a Galois extension with such a group.
\end{remark}
From the remark above we can conclude that the following proposition holds.
\begin{prop}
 For every $\ell \in \Z_{>0}$ and every map from $f: \{1,\ldots,\ell\} \times \{1,\ldots,\ell\} \longrightarrow \{1,\ldots,\ell\}$, there exists a first-order statement $\Gamma_{\ell,f}$  without parameters such that for any field $F$ we have that $F\models \Gamma_{\ell,f}$ if and only if $F$ has a Galois extension of size $\ell$ and $\phi$ represents the table of multiplication of the Galois group.
\end{prop}
This proposition has the following corollary which we will use below in the discussion of "Inverse Automorphism Problem".
\begin{cor}
\label{cor:invaut}
    For each pair $(n,m)$ with $n\equiv 0 \bmod m$ and a finite group $H$ of size $m$, there exists a first-order statement $\Sigma_{n,H}$ without parameters such that for any field $K$ we have that $K\models \Sigma_{n,H}$ if and only if there exists a positive integer $\ell \equiv 0 \bmod n$ with $n \leq \ell \leq n!$, a Galois extension $M/K$ of degree $\ell$ such that $\Gal(M/K)$ contains a subgroup $S$ of size $\ell/n$ with exactly $n/m$ distinct conjugates and the automorphism group of the field $L:=M^S$ over $K$ is isomorphic to $H$.
\end{cor}

\section{Translating generalizations of IGP into the first-order language of rings}
In this section, we consider two generalizations of IGP.  The first generalization is   a general version of \emph{Finite Split Embedding Problem} (FSEP), while the second one is the \emph{Inverse Automorphism Problem} (IAP).
\subsection{Translating FSEP into the first-order language of fields} The
Finite Split Embedding Problem FSEP$(E/F, G, \phi, \pi)$ takes the following input:
\begin{itemize}
\item a base field $F$ and a finite Galois extension $E/F$ with $H:=\Gal(E/F)$;
\item a finite group $G$ with an action $\phi: H \to \Aut(G)$, which gives rise to the semi-direct product $G \rtimes_\phi H$;
\item the semi-direct product $G \rtimes_\phi H$ comes equipped with an epimorphism $\pi:G \rtimes_\phi H \longrightarrow H$ of finite groups.
\end{itemize}
A solution to the FSEP is a Galois extension $L$ of $F$ that extends $E$, together with an isomorphism $i: \Gal(L/F) \xrightarrow{\sim} G \rtimes_{\phi}H$ such that $\pi \circ i = \textrm{res}_{E} : \Gal(L/F) \to \Gal(E/F)$.

 We also consider a \emph{generalized} version of FSEP, namely GFSEP$(F, H, G, \phi, \pi)$,  where the data consists of the base field $F$, the three groups $G$, $H$, $G \rtimes_\phi H$,  and the projection $\pi$.  A solution to GFSEP$(F, H, G, \phi, \pi)$ consists of a Galois field extension $E/F$ with $\Gal(E/F) \cong H$ such that FSEP$(E/F, G, \phi, \pi)$ is solvable.

Let us assume that we are given a base field $F$, the three finite groups $G$, $H$, $G \rtimes_\phi H$ (given by their table of multiplication),  and the projection $\pi$  as above. We will write down explicitly two sentences in the language of rings without parameters:
\begin{enumerate}
    \item A sentence $\mu_{H, G, \phi, \pi}$ such that $F \models \mu_{H, G, \phi, \pi}$ if and only if GFSEP$(F, H, G, \phi, \pi)$ has a solution.
    \item A sentence $\nu_{H, G,\phi,\pi,\bar a}$ such that $F\models \nu_{H, G,\phi,\pi,\bar a}$ if and only if $f_{\bar a}(T)$ is an irreducible polynomial over $F$ generating Galois extension $E/F$ such that FSEP$(E/F, G, \phi, \pi)$ has a solution.
\end{enumerate}
Before we begin our construction we describe notation to be used below:
\begin{enumerate}
    \item Given $G, H, \phi$, we have that the group $G\rtimes_{\phi} H$ is determined uniquely up to an isomorphism.  We assume that we are given the tables of multiplication for $G\rtimes_{\phi} H, G$ and $H$ as well as the epimorphism $\pi$.  
    \item If $f(T)$ is an irreducible polynomial over $F$ such that its root $\alpha$ generates a Galois extension $L$ over $F$ with Galois group isomorphic to  $G\rtimes_{\phi} H$,  then we denote its degree by $m:=|G\rtimes_{\phi} H|=|G||H|$.  
    
    \item Let $E$ be the fixed field of $G$ and let $g(T)$ be an irreducible polynomial over $F$ of degree $r:=|H|$ of some generator of $E$ over $F$.  Let $\beta$ be a root of $g(T)$. Let $\gamma$ be a generator of $L/E$. Observe that $\Gal(F(\alpha)/F(\beta))\cong G$ and $\Gal(F(\beta)/F) \cong H$ and $\pi: \Gal(F(\alpha)/F) \rightarrow \Gal(F(\beta)/F)$. 
\end{enumerate}

 To summarise, we have the following diagram:
 
\[
\xymatrix{
L=F(\alpha)=E(\gamma)  \\
E=F(\beta) \ar@{-}[u]_{G = \Gal(L/E)} \\
F \ar@{-}[u]_{H = \Gal(E/F)}\ar@{-}@/^3pc/{-}^{G \rtimes_{\phi} H}
}
\]

Using Proposition \ref{prop:together} we can now prove the following theorem.
\begin{theorem}
    There exists a first-order formula $\nu_{H, G, \phi, \pi}$ in the language of rings without parameters such that for any field $F$ we have that
    \[
    F \models \nu_{H, G, \phi, \pi}
    \]
     if and only if GFSEP$(F, H, G, \phi, \pi)$ is solvable.
\end{theorem}

\begin{proof}
We apply Proposition \ref{prop:together} with $S=G\rtimes_{\phi} H$.  Then GFSEP$(F, H, G, \phi, \pi)$ is solvable if and only if $F \models \exists \bar x \in F^r, \bar y_0 \in F^r,\ldots, \bar y_{m-1} \in F^r,\bar z \in F^s: \mu_{H,G,S,\pi}(\bar x, \bar y_0, \ldots,\bar y_{m-1}, \bar z)$.

\end{proof}
To construct a first-order statement corresponding to FSEP we need to take into consideration that the extension $E/F$ is fixed. So we choose polynomial $f_{\bar a}(T)$ such that any of its roots generate $E$ over $F$ and replace the first set of variables in $\mu_{H,G,S,\pi}$ by $\bar a$.
\begin{theorem}
    There exists a first-order formula $\delta_{G, \phi, \pi}$ in the language of rings without parameters such that for any Galois extension $E/F$  we have that
    \[
    F \models \nu_{H, G, \phi, \pi}
    \]
     if and only if FSEP$(E/F, H, G, \phi, \pi)$ is solvable, where $H= \Gal(E/F)$ is also given.
\end{theorem}
\begin{proof}
    We apply Proposition \ref{prop:together} with $S=G\rtimes_{\phi} H$ and $f_{\bar a}$ being a monic irreducible polynomial over $F$ of a generator of $E/F$.  Then GFSEP$(E/F, G, \phi, \pi)$ is solvable if and only if $F \models \exists  \bar y_0 \in F^r,\ldots, \bar y_{m-1} \in F^r,\bar z \in F^s: \mu_{H,G,S,\pi}(\bar a, \bar y_0, \ldots,\bar y_{m-1}, \bar z)$, where $H=\Gal(E/F)$.
\end{proof}


\subsection{Translating the Inverse Automorphism Problem into the first-order language of the fields} 
The Inverse Automorphism Problem (IAP), suggested to us by Arno Fehm, takes the following input:
\begin{itemize}
\item a base field $K$;
\item a finite group $H$, say of size $|H|=:m$;
\item an positive integer $n$ which is a multiple of $m$.
\end{itemize}
A solution to the IAP is a finite extension $L$ of $K$ of degree $n$ (not necessarily Galois over $K$) with automorphism group over $K$ isomorphic to $H$. 

We prove the following result.
\begin{theorem}
\label{thm:conj}
 Let $H$ be a finite group of size $m$.  Let $n$ be a multiple of $m$.  Then there exists a first-order statement $\varepsilon_{H,n}(x_1,\ldots,x_n)$ in the language of rings without parameters depending on $H$ and $n$ only such that $F\models \exists \bar x \in F^n: \varepsilon_{H,n}(x_1,\ldots,x_n)$ if and only if there exists an extension $L/F$ of degree $n$ such $\Aut(L/F) \cong H$.

\end{theorem}

To prove the theorem we use the following proposition.

\begin{prop}
Let $m, n$ be given positive integers with $n \equiv 0 \bmod m$.   Let $K$ be a field, let $M$ be a Galois extension of $K$ of degree $\ell$, and $n \leq \ell \leq n!$. A degree $n$ extension $L$ of $K$ inside $M$ with exactly $m$ automorphisms exists if and only if  there exists a subgroup $S$ of $G := \Gal(M/K)$ of index $n$ with exactly $n/m$ conjugates in $G$. In this case, $L = M^S$, the fixed field of $S$ in $M$.
\end{prop}

\begin{proof}
Assume first that there exists a Galois extension $M$ of $K$ of degree $\ell$ such that $\ell \equiv 0 \bmod n$ with $\ell \leq n!$, and $G=\Gal(M/K)$ contains a subgroup $S$ of $G$ of size $\ell/n$ with exactly $n/m$ conjugates in $G$. 
Then we claim that $L = M^S$ gives the desired extension. 

Clearly, $[L:K]=\ell/|S|=n$.  We now show that $L$ has $m$ automorphisms over $K$. Let $S_1=S, \ldots, S_{n/m}$ be all of the distinct conjugates of $S$ in $G$.  If $S_i=\sigma_iS\sigma_i^{-1}$ for some $\sigma_i \in \Gal(M/K)$, then $\sigma_1(L)=L,\ldots, \sigma_{n/m}(L)$ are distinct conjugates of $L$ over $K$ in $M$ since $S_i=\Gal(M/\sigma_i(L))$. Further, if $\hat L=\sigma(L)\ne L$ is a conjugate of $L$, then $\sigma S\sigma^{-1}=\hat S=\Gal(M/\hat L)$ and $\hat S \ne S$ is a conjugate of $S$. Thus, there are exactly $n/m$ conjugate subfields $\hat L$ of $L$ in $M$. 
\[
\xymatrix{
M \\
L \ar@{-}[u]_{|S|} \ar[r]^{\sigma} & \hat{L} \ar@{-}[dl] \ar@{-}@/_1.5pc/{-}_{|\sigma S \sigma^{-1}| = |S|}\\
K \ar@{-}[u]_{\ell/|S| =n}\ar@{-}@/^2pc/{-}^{|\text{Gal}(M/K)|=\ell}
}
\]

Fix a $\sigma \in G$ with an induced isomorphism $L \to \hat L$. If $\phi \in G$ induces an automorphism of $L$, then $\sigma \circ \phi$ restricts to an isomorphism $L \to \hat L$ as well. 
Conversely, given any $\tau\in G$ inducing an isomorphism $L \to \hat L$, we get a $K$-automorphism of $L$ via $\tau^{-1}|_{\hat L}\circ \sigma|_L: L \to L$.

Now, we know that there are $n = [L:K]$ embeddings of $L$ into $M$, and by the above argument, these automorphisms can be broken up as
\begin{align}
\label{eqn:embeddings}
n & = \#\{K\textup{-automorphisms of }L\} \cdot \#\{\textup{conjugate fields $\hat L$ of $L$ in $M$}\}\\
& =
\#\{K\textup{-automorphisms of }L\} \cdot \#\{\textup{conjugate groups $\hat S$ of $S$ in $G$}\}.
\end{align}
Therefore, we have that the number of distinct $K$-automorphisms of $L$ is equal to $n/(n/m)=m$, as required.

Conversely, suppose now that $K\subseteq L \subseteq M$, where $M/K$ is Galois, satisfies $[L:K] = n$ and $L$ has exactly $m$ automorphisms. Let $S=\Gal(M/L)$. We assume that $m <n$, since otherwise  $L/K$ is Galois, and  we are in the case considered in the preceding section.  Since $m < n$, the extension $L/K$ is not Galois, and therefore $S=\Gal(M/L)$ is not normal. Now we use Equation \eqref{eqn:embeddings} to conclude that the number of conjugates of $S$ is $n/m$, as required.
\end{proof}

Now the proof of Theorem \ref{thm:conj} follows from Corollary \ref{cor:invaut}.



\end{document}